\documentclass[10 pt,reqno]{amsart}
\theoremstyle{definition}
\usepackage{indentfirst, amssymb,amsmath,amsthm}
\usepackage{csquotes}
\usepackage{hyperref}
\newtheorem*{theoA}{Theorem A}
\newtheorem*{theoB}{Theorem B}
\newtheorem*{theoC}{Theorem C}
\newtheorem*{theoD}{Theorem D}
\newtheorem*{theoE}{Theorem E}
\newtheorem*{theoF}{Theorem F}
\newtheorem*{theoG}{Theorem G}
\newtheorem*{theoH}{Theorem H}
\newtheorem*{theoI}{Theorem I}

\newtheorem{theo}{Theorem}[section]
\newtheorem{lem}{Lemma}[section]

\newtheorem{defi}{Definition}[section]

\newcommand{\be}{\begin{equation}}
\newcommand{\ee}{\end{equation}}
\newcommand{\beas}{\begin{eqnarray*}}
\newcommand{\eeas}{\end{eqnarray*}}
\newcommand{\bea}{\begin{eqnarray}}
\newcommand{\eea}{\end{eqnarray}}

\numberwithin{equation}{section}
\begin{document}
\title[Value distribution of some D-D monomials]{A Note on the value distribution of some differential- difference monomials generated by a transcendental entire function of hyper-order less than one}
\date{}
\author[S. Roy and et al]{Soumon Roy$^{1}$ and  Sudip Saha$^{2}$ and Ritam Sinha$^{3}$}
\date{}
\address{$^{1}$ \textbf{Soumon Roy}, Nevannlina Lab, Ramakrishna Mission Vivekananda Centenary College, Rahara,
West Bengal 700 118, India.}
\email{rsoumon@gmail.com}
\address{$^{2}$ \textbf{Sudip Saha}, Department of Mathematics, Brainware University, Barasat, kolkata, West Bengal 700 125, India.}
\email{sudipsaha814@gmail.com}
\address{$^{3}$ \textbf{Ritam Sinha}, Nevannlina Lab, Ramakrishna Mission Vivekananda Centenary College, Rahara,
West Bengal 700 118, India.}
\email{sinharitam23@gmail.com}
\maketitle
\let\thefootnote\relax
\footnotetext{2020 AMS Mathematics Subject Classification: 30D30, 30D20, 30D35, 30D45.}
\footnotetext{Key words and phrases: Hyper order, Meromorphic function, Shift, Monomials}
\begin{abstract}{}
Let $\mathfrak{f}$ be a transcendental entire function with hyper-order less than one. The aim of this note is to study the value distribution of the differential-difference monomials $\alpha \mathfrak{f}(z)^{q_0}(\mathfrak{f}(z+c))^{q_1}$, where $c$ is a non-zero complex number and $q_0\geq2,$ $q_1\geq 1$ are non-negative integers, and  $ \alpha(z)$ $(\not\equiv 0,\infty)$ be a small function of $\mathfrak{f}$.
\end{abstract}
\section{Introduction}
Throughout this paper, we assume that the reader is familiar with the value distribution theory of meromorphic functions and is acquainted with the standard notations of Nevanlinna theory (\cite{05}). Furthermore, let $E$ denote any set of positive real numbers with finite Lebesgue measure, where $E$ may vary in each occurrence.\par
A meromorphic function $\mathfrak{f}$ is either analytic or contains at most countable numbers of poles in the complex plane. If  $\mathfrak{f}$ has no poles, then $\mathfrak{f}$ reduces to an entire function. The order of $\mathfrak{f}$ is denoted by $\rho(\mathfrak{f})$ and the hyper-order of $\mathfrak{f}$ is denoted by $\rho_2(\mathfrak{f})$ and defined as
 \begin{equation*}
\rho_{2}(\mathfrak{f}):=\limsup_{r\to\infty}\frac{\log \log T(r,f)}{\log r}.
\end{equation*}
For any non-constant meromorphic function $\mathfrak{f}$, we denote by $S(r,\mathfrak{f})$, any quantity satisfying $$ S(r,\mathfrak{f})=o(T(r,\mathfrak{f}))~~~\text{as}~~ r \rightarrow \infty, r \notin E. $$
By \emph{shift} of $\mathfrak{f}(z)$, we mean $\mathfrak{f}(z+c)$, where $c$ is a non-zero complex constant. Let $\mathfrak{f}$ be a non-constant meromorphic function. A meromorphic function $a(z)$ $(\not \equiv 0, \infty)$ is called a ``small function" with respect to $\mathfrak{f}$ if $T(r,a(z))= S(r,\mathfrak{f})$.
\begin{defi} (\cite{14})
Let $ a\in \mathbb{C} \cup \{\infty\}$. For a positive integer $k$, we define
\begin{enumerate}
\item[(i)] $N_{k)}(r,a; \mathfrak{f})$ as the counting function of $a$- points of $\mathfrak{f}$ with multiplicity $\leq k$,\\
\item[(ii)] $N_{(k}(r,a; \mathfrak{f})$ as the counting function of $a$- points of $\mathfrak{f}$ with multiplicity $\geq k$.
\end{enumerate}
\end{defi}
Similarly, the reduced counting function $\overline{N}_{k)}(r,a; \mathfrak{f})$ and $\overline{N}_{(k}(r,a; \mathfrak{f})$ are defined.
\begin{defi}(\cite{08})
Let $k\in \mathbb{N}$. We define $N_{k}(r,0; \mathfrak{f})$ as the counting function of zeros of $\mathfrak{f}$, where a zero of $\mathfrak{f}$ with multiplicity $q$ is counted $q$ times if $ q\leq k$, and is counted $k$ times if $q>k$.
\end{defi}
In 1959, Hayman (\cite{04}) proved the following theorem:
\begin{theoA} (\cite{04})
If $\mathfrak{f}$ is a transcendental meromorphic function and $n\geq 3 $, then $\mathfrak{f}^n\mathfrak{f}'$ assumes all finite values except possibly zero infinitely often.
\end{theoA}
Moreover, Hayman (\cite{04}) conjectured that the Theorem A remains valid for the cases $n= 1$, $2$. In 1979, Mues (\cite{09}) confirmed the Hayman's conjecture for $n=2$ as well as Chen and Fang (\cite{02}) ensured the conjecture for $n=1$ in 1995. Later, in 1992, Q. Zhang (\cite{15}) gave the quantitative version of Mues's result as follows:
\begin{theoB} (\cite{15})
For a transcendental meromorphic function $\mathfrak{f}$, the following inequality holds:
$$ T(r,\mathfrak{f}) \leq 6 N\left(r,\frac{1}{\mathfrak{f}^2\mathfrak{f}'-1}\right) + S(r,\mathfrak{f}).$$
\end{theoB}
  In (\cite{11}), Theorem B was improved by Xu and Yi as
\begin{theoC} (\cite{11})
For a transcendental meromorphic function $\mathfrak{f}$ and $ \phi(z)$ $(\not \equiv 0, \infty)$ be a small function, then the following inequality holds:
$$ T(r,\mathfrak{f}) \leq 6 \overline{N}\left(r,\frac{1}{\phi \mathfrak{f}^2\mathfrak{f}'-1}\right) + S(r,\mathfrak{f}).$$
\end{theoC}
Later, in 2005, Huang and Gu (\cite{06}) replaced $\mathfrak{f}'$ by $\mathfrak{f}^{(k)}$, where $k(\geq1)$ is an integer and proved the following result:
\begin{theoD} (\cite{06})
For a transcendental meromorphic function $\mathfrak{f}$ and  a positive integer $k$, the following inequality holds:
$$ T(r,\mathfrak{f}) \leq 6 N\left(r,\frac{1}{\mathfrak{f}^2\mathfrak{f}^{(k)}-1}\right) + S(r,\mathfrak{f}).$$
\end{theoD}
A natural question arises as to whether the given inequality remains valid when the counting function in Theorem D is replaced by the reduced counting function. Addressing this question, Xu, Yi, and Zhang made significant progress in 2009 by proving the following theorem (\cite{12}):
\begin{theoE} (\cite{12})
Let $\mathfrak{f}$ be a transcendental meromorphic function, and $k$ $(\geq1)$ be a positive integer. If $N_1(r,0;\mathfrak{f})=S(r,\mathfrak{f})$, then
$$ T(r,\mathfrak{f}) \leq 2 \overline{N}\left(r,\frac{1}{\mathfrak{f}^2\mathfrak{f}^{(k)}-1}\right) + S(r,\mathfrak{f}).$$
\end{theoE}
Later, in 2011, removing the restrictions on zeros of $\mathfrak{f}$, Xu, Yi and Zhang (\cite{13}) proved the following theorem:
\begin{theoF} (\cite{13})
Let $\mathfrak{f}$ be a transcendental meromorphic function, and $k$ $(\geq1)$ be a positive integer. Then
$$ T(r,\mathfrak{f}) \leq M \overline{N}\left(r,\frac{1}{\mathfrak{f}^2\mathfrak{f}^{(k)}-1}\right) + S(r,\mathfrak{f}),$$
 where $M $ is $6$ if $k=1$, or $k\geq3$ and $M$ is $10$ if $k=2$.
\end{theoF}
Later, in 2018, Karmakar and Sahoo (\cite{07}) further improved the Theorem F and obtained the following result:
\begin{theoG} (\cite{07})
Let $\mathfrak{f}$ be a transcendental meromorphic function, and $k$ $(\geq1)$, $n$ $(\geq2)$ be a positive integer. Then
$$ T(r,\mathfrak{f}) \leq \frac{6}{2n-3} \overline{N}\left(r,\frac{1}{\mathfrak{f}^n\mathfrak{f}^{(k)}-1}\right) + S(r,\mathfrak{f}).$$
\end{theoG}
Furthermore, in 2020,  Chakraborty et al. (\cite{01}) generalized Theorem G for transcendental meromorphic functions without simple poles and established the following result:
\begin{theoH} (\cite{01})
Let $\mathfrak{f}$ be a transcendental meromorphic function such that it has no simple pole. Also let $q_0$ $(\geq2)$, $q_k$ $(\geq1)$ are $(k \in \mathbb{N})$ integer. Then
$$ T(r,\mathfrak{f}) \leq \frac{6}{2q_{0}-3} \overline{N}\left(r,\frac{1}{\mathfrak{f}^{q_0}{(\mathfrak{f}^{(k)})}^{q_k}-1}\right) + S(r,\mathfrak{f}).$$
\end{theoH}
Later, Bhoosnurmath, Chakraborty, and Srivastava (\cite{0}), as well as Chakraborty and L\"{u} (\cite{1}), and Saha and Chakraborty (\cite{10}), obtained many significant results in this direction. Among these, the work of Saha and Chakraborty (\cite{10}) notably improved and extended the results of Karmakar and Sahoo (\cite{07}) for a specific class of transcendental meromorphic functions. Furthermore, the following theorem represents a significant improvement over the recent results of B. Chakraborty et al. (\cite{01}).
\begin{theoI} (\cite{10})
Let $\mathfrak{f}$ be a transcendental meromorphic function such that $N_{1)}(r,\infty;\mathfrak{f})=S(r,\mathfrak{f})$ and $\alpha=\alpha(z)$ $(\not \equiv 0,\infty)$ be a small function of $\mathfrak{f}$. Also let $q_0$ $(\geq2)$, $q_i$ $(\geq0)$ $(i=1,2,\ldots,k-1)$, $q_k$ $(\geq1)$ are non- negative integers. Then for any small function  $a=a(z)(\not \equiv 0,\infty)$ of $\mathfrak{f}$,
$$ T(r,\mathfrak{f}) \leq \frac{2}{2q_{0}-3} \overline{N}\left(r,\frac{1}{\alpha\mathfrak{f}^{q_0}{(\mathfrak{f}')}^{q_1}\ldots {(\mathfrak{f}^{(k)})}^{q_k} -a}\right) + S(r,\mathfrak{f}).$$
\end{theoI}
Motivated by Theorem H and Theorem I, in this note, we are try to investigate the value distribution of a differential- difference monomials $\mathfrak{f}(z)^{q_0}(\mathfrak{f}(z+c))^{q_1}$ generated by a transcendental entire function with hyper-order less than one.
\section{main results}
\begin{theo}\label{th1}
Let $\mathfrak{f}$ be a transcendental entire function with hyper-order less than one and $ \alpha(z)$ $(\not\equiv 0,\infty)$ be a small function of $\mathfrak{f}$. Let $c$ be a non-zero complex number and $q_0\geq2,$ $q_1\geq 1$ be non-negative integers. Then for any small function $a(z)$ $(\not\equiv 0,\infty)$ of $\mathfrak{f}$,
$$ T(r,\mathfrak{f}) \leq  \frac{1}{\left(q_{0}-1\right)} \overline{N}\left(r, \frac{1}{\alpha(z) \mathfrak{f}(z)^{q_0}(\mathfrak{f}(z+c))^{q_1}-a(z)}\right)+ S(r,\mathfrak{f}).$$
\end{theo}
\begin{theo}\label{th2}
Let $\mathfrak{f}$ be a transcendental entire function with hyper-order less than one and $ \alpha(z)$ $(\not\equiv 0,\infty)$ be a small function of $\mathfrak{f}$. Let $c$ be a non-zero complex number and $q_1\geq 1$ and $q_0> q_1+1$ be non-negative integers. Then for any small function  $a(z)(\not\equiv 0,\infty)$ of $\mathfrak{f}$,
$$ T(r,\mathfrak{f}) \leq  \frac{1}{\left(q_{0}-q_{1}-1\right)} \overline{N}\left(r, \frac{1}{\alpha(z) \mathfrak{f}(z)^{q_0}(\mathfrak{f}'(z+c))^{q_1}-a(z)}\right)+ S(r,\mathfrak{f}).$$
\end{theo}
\section{Some important Lemmas}
To establish our results, we require the following lemmas:
\begin{lem}\label{3.1} (\cite{03})
Let $\mathfrak{f}(z)$ be a non-constant transcendental meromorphic function with hyper-order less than one, and let $c$ be a non-zero finite complex number. Then $$ m\left(r,\frac{\mathfrak{f}(z+c)}{\mathfrak{f}(z)}\right)=S(r,\mathfrak{f})~~\text{and}~~ m\left(r,\frac{\mathfrak{f}(z)}{\mathfrak{f}(z+c)}\right)=S(r,\mathfrak{f}).$$
\end{lem}
\begin{lem}\label{3.2}(\cite{16}, Lemma 6)
Let $\mathfrak{f}(z)$ be a meromorphic function with hyper-order less than one, and let $c$ be a non-zero finite complex number. Then
$$T(r,\mathfrak{f}(z+c)) =T(r,\mathfrak{f})+ S(r,\mathfrak{f})~~\text{and}~~N(r,\mathfrak{f}(z+c)) =N(r,\mathfrak{f})+ S(r,\mathfrak{f}).$$
\end{lem}
\begin{lem}\label{3.3}
Let $\mathfrak{f}$ be a transcendental meromorphic function with hyper-order less than one. Then $$m\left(r,\frac{{\mathfrak{f}}^{(k)}(z+c)} {\mathfrak{f}(z+d)}\right)= S(r,\mathfrak{f}),$$ for all $z$ satisfies $|z|=r\notin E$, $E$ is a set with  finite logarithmic measure, where $c$ and $d$ are complex constants and $k$ is a non-negative  integer.
\end{lem}
\begin{proof} In view of Lemma \ref{3.1}, we have
\begin{eqnarray*}
m\left(r,\frac{{\mathfrak{f}}^{(k)}(z+c)} {\mathfrak{f}(z+d)}\right) &\leq & m\left(r,\frac{{\mathfrak{f}}^{(k)}(z+c)} {\mathfrak{f}(z+c)}\right)+ m\left(r,\frac{\mathfrak{f}(z+c)} {\mathfrak{f}(z)}\right)+  m\left(r,\frac{\mathfrak{f}(z)} {\mathfrak{f}(z+d)}\right)\\
& \leq & S(r,\mathfrak{f}).
\end{eqnarray*}
Thus, the lemma is proved.
\end{proof}
\begin{lem}\label{3.4}
Let $\mathfrak{f}$ be a transcendental entire function with hyper-order less than one and $ \alpha(z)$ $(\not\equiv 0,\infty)$ be a small function of $\mathfrak{f}$. Let $ M_1[\mathfrak{f}]= \alpha(z) \mathfrak{f}(z)^{q_0}(\mathfrak{f}(z+c))^{q_1}$, where $c$ is a non-zero complex number and $q_0\geq2,$ $q_1\geq 1$ are integers. Then $$T(r,M_1[\mathfrak{f}]) \leq (q_0 + q_1) T(r,\mathfrak{f})+ S(r,\mathfrak{f}).$$
\end{lem}
\begin{proof}
In view of Lemma (\ref{3.2}), the proof is obvious. \end{proof}
\begin{lem}\label{3.6}
Let $\mathfrak{f}$ be a transcendental entire function with hyper-order less than one and $ \alpha(z)$ $(\not\equiv 0,\infty)$ be a small function of $\mathfrak{f}$. Let $ M_1[\mathfrak{f}]= \alpha \mathfrak{f}(z)^{q_0}(\mathfrak{f}(z+c))^{q_1}$, where $c$ is a non-zero complex number and $q_0\geq2,$ $q_1\geq 1$ are non-negative integers. Then $M_1[\mathfrak{f}]$ is not identically constant.
\end{lem}
\begin{proof}
Define $\mu= q_0+q_1$ be degree of the monomial $M_1[\mathfrak{f}]$. Note that $$ \frac{S(r, M_1[\mathfrak{f}])}{T(r,\mathfrak{f})}=  \frac{S(r, M_1[\mathfrak{f}])}{T(r,M_1[\mathfrak{f}])}\cdot\frac{T(r,M_1[\mathfrak{f}])}{T(r,\mathfrak{f})}.$$
Thus we can conclude that $ S(r, M_1[\mathfrak{f}])= S(r, \mathfrak{f}).$
Since $$\frac{1}{\mathfrak{f}^{\mu}} = \alpha(z) \left(\frac{\mathfrak{f}(z+c)}{\mathfrak{f}(z)}\right)^{q_1}\frac{1}{M_1[\mathfrak{f}]},$$
therefore, by the 1st fundamental theorem and Lemma (\ref{3.1}), we have
\begin{eqnarray*}{}
\mu T\left(r,\mathfrak{f}\right) & \leq & q_1 T\left(r,\frac{\mathfrak{f}(z+c)}{\mathfrak{f}(z)}\right) + T\left(r,\frac{1}{M_1[\mathfrak{f}] }\right) + S(r, \mathfrak{f})\\
 & \leq &  q_1 N\left(r,\frac{\mathfrak{f}(z+c)}{\mathfrak{f}(z)}\right) + T\left(r,M_1[\mathfrak{f}]\right) + S(r, \mathfrak{f})\\
& \leq & q_1 [N(r,0;\mathfrak{f}) + N(r,\infty ;\mathfrak{f}(z+c))]+T\left(r,M_1[\mathfrak{f}]\right) + S(r, \mathfrak{f})\\
& \leq & q_1 N(r,0;\mathfrak{f}) +T\left(r,M_1[\mathfrak{f}]\right) + S(r, \mathfrak{f})\\
& \leq & (q_1+1) T(r,M_1[\mathfrak{f}]) + S(r, \mathfrak{f})
\end{eqnarray*}
Therefore $M_1[\mathfrak{f}]$ is not identically constant.
\end{proof}
\begin{lem}\label{3.7}
Let $\mathfrak{f}$ be a transcendental entire function with hyper-order less than one and $ \alpha(z)$ $(\not\equiv 0,\infty)$ be a small function of $\mathfrak{f}$. Let $ M_2[\mathfrak{f}]= \alpha(z) \mathfrak{f}(z)^{q_0}(\mathfrak{f}'(z+c))^{q_1}$, where $c$ is a non-zero complex number and $q_0\geq2,$ $q_1\geq 1$ are non-negative integers. Then $M_2[\mathfrak{f}]$ is not identically constant.
\end{lem}
\begin{proof}
The proof is similar to the proof of Lemma \ref{3.6}.
\end{proof}
\begin{lem}\label{3.9}
Let $\mathfrak{f}$ be a transcendental entire function with hyper-order less than one and $ \alpha(z)$ $(\not\equiv 0,\infty)$ be a small function of $\mathfrak{f}$. Let $ M_1[\mathfrak{f}]= \alpha(z) \mathfrak{f}(z)^{q_0}(\mathfrak{f}(z+c))^{q_1}$, where $c$ is a non-zero complex number and $q_0\geq2,$ $q_1\geq 1$ are non-negative integers. If $a(z)$ $(\not\equiv 0,\infty)$ is an another small function of $\mathfrak{f}$, then
 $$\mu T(r,\mathfrak{f}) \leq  N(r,0;\mathfrak{f}^{\mu}) + \overline{N}(r,0;M_1[\mathfrak{f}]) + \overline{N}(r,a;M_1[\mathfrak{f}]) - N(r,0;M_1[\mathfrak{f}]) + S(r,\mathfrak{f}),$$
 where $\mu= q_0+q_1$.
\end{lem}
\begin{proof}
Now, with help of Lemma \ref{3.1}, we have
\begin{eqnarray}\label{eq2d}
\nonumber\mu T(r,\mathfrak{f})= T\left(r, \frac{1}{\mathfrak{f}^{\mu}}\right)+ O(1) & = & N(r,0; \mathfrak{f}^{\mu})+ m(r,0; \mathfrak{f}^{\mu})+ S(r,\mathfrak{f})\\
\nonumber& \leq  & N(r,0; \mathfrak{f}^{\mu})+ m(r,0; M_1[\mathfrak{f}])+ S(r,\mathfrak{f})\\
& \leq  & N(r,0; \mathfrak{f}^{\mu})+T(r, M_1[\mathfrak{f}]) - N(r,0; M_1[\mathfrak{f}])+ S(r,\mathfrak{f})
\end{eqnarray}
Now, using the Nevanlinna's 2nd Fundamental theorem for the function $M_1[\mathfrak{f}]$, we have
\begin{equation}\label{eq2}
T(r, M_1[\mathfrak{f}]) \leq \overline{N}(r,\infty; M_1[\mathfrak{f}])+ \overline{N}(r,0; M_1[\mathfrak{f}]) + \overline{N}(r,a; M_1[\mathfrak{f}]) + S(r,M_1[\mathfrak{f}])
\end{equation}
Using (\ref{eq2d}) and (\ref{eq2}), we have
$$\mu T(r,\mathfrak{f}) \leq  N(r,0;\mathfrak{f}^{\mu}) + \overline{N}(r,0;M[\mathfrak{f}]) + \overline{N}(r,a;M[\mathfrak{f}]) - N(r,0;M[\mathfrak{f}]) + S(r,\mathfrak{f}).$$
Thus, the lemma is proved.
\end{proof}
\begin{lem}\label{3.10}
Let $\mathfrak{f}$ be a transcendental entire function with hyper-order less than one and $ \alpha(z)$ $(\not\equiv 0,\infty)$ be a small function of $\mathfrak{f}$. Let   $ M_2[\mathfrak{f}]= \alpha(z) \mathfrak{f}(z)^{q_0}(\mathfrak{f}'(z+c))^{q_1}$, where $c$ is a non-zero complex number and $q_0\geq2,$ $q_1\geq 1$  are non-negative integers. If $a(z)$ $(\not\equiv 0,\infty)$ is a small function of $\mathfrak{f}$, then
 $$\mu T(r,\mathfrak{f}) \leq  N(r,0;\mathfrak{f}^{\mu}) + \overline{N}(r,0;M_2[\mathfrak{f}]) + \overline{N}(r,a;M_2[\mathfrak{f}]) - N(r,0;M_2[\mathfrak{f}]) + S(r,\mathfrak{f}),$$
 where $\mu= q_0+q_1$.
\end{lem}
\begin{proof}
The proof is similar to the proof of the Lemma \ref{3.9}.
\end{proof}
\section{Proofs of the Main Theorems}
\begin{proof}[\textbf{Proof of the Theorem \ref{th1}}]
Let $ M_1[\mathfrak{f}]= \alpha(z) \mathfrak{f}(z)^{q_0}(\mathfrak{f}(z+c))^{q_1}$ and $\mu= q_0+q_1$. \medbreak
\textbf{Case 1:} Let $z_0$ be a zero of $\mathfrak{f}(z)$ of multiplicity $q$ and $\mathfrak{f}(z+c)$ has no zero at $z_0$.\medbreak
Thus  $z_0$ is a zero of $\mathfrak{f}(z)^{q_0}(\mathfrak{f}(z+c))^{q_1}$ of order at least $ qq_0.$  Now,
$$ q\mu +1 -qq_0 = q(q_0+q_1) +1 -qq_0  \leq  qq_1 +1$$
Therefore,
\begin{eqnarray}\label{p1}
\nonumber && N(r,0;\mathfrak{f}^{\mu}) + \overline{N}(r,0;M_1[\mathfrak{f}]) - N(r,0;M_1[\mathfrak{f}]) \\
 &\leq & q_1 N(r,0;\mathfrak{f}) + \overline{N}(r,0;\mathfrak{f})
\end{eqnarray}
Using Lemma (\ref{3.9}) and (\ref{p1}), we get
\begin{eqnarray*}
\mu T(r, \mathfrak{f}) & \leq &  \overline{N}(r,a; M_1[\mathfrak{f}])+ (q_1+1) T(r,\mathfrak{f})+ S(r,\mathfrak{f}).
\end{eqnarray*}
Therefore, $$T(r, \mathfrak{f})  \leq  \frac{1}{(q_0-1)}\overline{N}(r,a;M_1[\mathfrak{f}])+ S(r,\mathfrak{f}).$$
\textbf{Case 2: } Let $z_0$ be a zero of $\mathfrak{f}(z)$ of multiplicity $q$ and $\mathfrak{f}(z+c)$ has also zero at $z_0$ of multiplicity $k$.\medbreak
If $q=k$, then $z_0$ is a zero of $\mathfrak{f}(z)^{q_0}(\mathfrak{f}(z+c))^{q_1}$ of order at least $ qq_0 + kq_1.$ Now,
$$ q\mu +1 -(qq_0+ kq_1)= q(q_0+q_1) +1 -qq_0-kq_1\leq 1$$
Therefore,
\begin{eqnarray}\label{s1}
 N(r,0;\mathfrak{f}^{\mu}) + \overline{N}(r,0;M_1[\mathfrak{f}]) - N(r,0;M_1[\mathfrak{f}])\leq  \overline{N}(r,0;\mathfrak{f})
\end{eqnarray}
Using (\ref{s1}) and Lemma (\ref{3.9}), we get
\begin{eqnarray*}
\mu T(r, \mathfrak{f}) & \leq &  \overline{N}(r,a;M[\mathfrak{f}])+ T(r, \mathfrak{f})+ S(r,\mathfrak{f})
\end{eqnarray*}
Since $ q_1 \geq 1$, then
 $$T(r, \mathfrak{f})  \leq  \frac{1}{(q_0-1)}\overline{N}(r,a;M[\mathfrak{f}])+ S(r,\mathfrak{f}).$$
If $q\not=k$, then  $z_0$ is a zero of $\mathfrak{f}(z)^{q_0}(\mathfrak{f}(z+c))^{q_1}$ of order at least $qq_0 + kq_1.$ Now,
$$ q\mu +1 -(qq_0+ kq_1) = qq_1 +1 -kq_1 \leq  qq_1 +1 $$
Therefore,
\begin{eqnarray}\label{e1}
\nonumber&& N(r,0;\mathfrak{f}^{\mu}) + \overline{N}(r,0;M[\mathfrak{f}]) - N(r,0;M[\mathfrak{f}]) \\
 &\leq &  q_1 N(r,0;\mathfrak{f}) + \overline{N}(r,0;\mathfrak{f})
\end{eqnarray}
Using (\ref{e1}) and Lemma (\ref{3.9}), we get
\begin{eqnarray*}
\mu T(r, \mathfrak{f}) & \leq &  \overline{N}(r,a;M[\mathfrak{f}])+ q_1 T(r,\mathfrak{f}) +T(r, \mathfrak{f})+ S(r,\mathfrak{f}).
\end{eqnarray*}
Therefore, $$T(r, \mathfrak{f})  \leq  \frac{1}{(q_0-1)}\overline{N}(r,a;M[\mathfrak{f}])+ S(r,\mathfrak{f}).$$

\textbf{Case 3:} Let $\mathfrak{f}(z)$ has no zero at $z_0$, but  $\mathfrak{f}(z+c)$ has zero at $z_0$ of multiplicity $q$.\medbreak
Then  $z_0$ is a zero of $\mathfrak{f}(z)^{q_0}(\mathfrak{f}(z+c))^{q_1}$ of order at least $ qq_1.$ Now,
$$ 1 -qq_1  \leq  1  \leq  qq_1 +1$$
Therefore,
\begin{eqnarray}
\label{t1}\nonumber && N(r,0;\mathfrak{f}^{\mu}) + \overline{N}(r,0;M[\mathfrak{f}]) - N(r,0;M[\mathfrak{f}]) \\
 &\leq & q_1 N(r,0;\mathfrak{f}(z+c)) + \overline{N}(r,0;\mathfrak{f}(z+c))
\end{eqnarray}
Using (\ref{t1}), Lemma (\ref{3.9}) and Lemma (\ref{3.2}), we get
\begin{eqnarray*}
\mu T(r, \mathfrak{f}) & \leq &  \overline{N}(r,a;M[\mathfrak{f}])+ q_1 T(r,0;\mathfrak{f}(z+c))+ T(r,0;\mathfrak{f}(z+c))+ S(r,\mathfrak{f})\\
& \leq &  \overline{N}(r,a;M[\mathfrak{f}])+ q_1 T(r,\mathfrak{f})+ T(r,\mathfrak{f})+ S(r,\mathfrak{f})
\end{eqnarray*}
Therefore, $$T(r, \mathfrak{f})  \leq  \frac{1}{(q_0-1)}\overline{N}(r,a;M[\mathfrak{f}])+ S(r,\mathfrak{f}).$$
This completes the proof.
\end{proof}
\begin{proof}[\textbf{Proof of the Theorem \ref{th2}}]
Let $ M_2[\mathfrak{f}]= \alpha(z) \mathfrak{f}(z)^{q_0}(\mathfrak{f}'(z+c))^{q_1}$ and $\mu= q_0+q_1$. \medbreak
\textbf{Case 1}:  Let $z_0$ be a zero of $\mathfrak{f}(z)$ of multiplicity $q$, and $\mathfrak{f}(z+c)$ has also zero at $z_0$ of multiplicity $k$.\medbreak
If  $q\neq k$, then $z_0$ is a zero of $\mathfrak{f}(z)^{q_0}(\mathfrak{f}'(z+c))^{q_1}$ of order at least $ qq_0 + (k-1)q_1.$ Now,\\ $$ q\mu +1 -(qq_0+ (k-1)q_1) \leq  qq_1+1+q_1 $$
Therefore,
\begin{eqnarray}\label{01}
\nonumber&& N(r,0;\mathfrak{f}^{\mu}) + \overline{N}(r,0;M[\mathfrak{f}]) - N(r,0;M[\mathfrak{f}]) \\
 &\leq &  q_1N(r,0;\mathfrak{f})+ \overline{N}(r,0;\mathfrak{f})+ q_1\overline{N}(r,0;\mathfrak{f})
\end{eqnarray}
Using Lemma (\ref{3.10}) and (\ref{01}), we get
\begin{eqnarray*}
\mu T(r, \mathfrak{f}) & \leq &  \overline{N}(r,a;M[\mathfrak{f}])+  (2q_1+1)T(r,\mathfrak{f})+ S(r,\mathfrak{f})
\end{eqnarray*}
Since  $q_0> q_1+1$, then
 $$T(r, \mathfrak{f})  \leq  \frac{1}{(q_0-q_1-1)}\overline{N}(r,a;M[\mathfrak{f}])+ S(r,\mathfrak{f}).$$
If $q=k$, then  $z_0$ is a zero of $\mathfrak{f}(z)^{q_0}(\mathfrak{f}'(z+c))^{q_1}$ of order at least $ qq_0 + (k-1)q_1.$ Now,
$$ q\mu +1 -(qq_0+ (k-1)q_1)=q(q_0+q_1) +1 -qq_0-kq_1+q_1 \leq  q_1 +1$$
Therefore,
\begin{eqnarray} \label{02}
N(r,0;\mathfrak{f}^{\mu}) + \overline{N}(r,0;M[\mathfrak{f}]) - N(r,0;M[\mathfrak{f}])\leq   q_1 \overline{N}(r,0;\mathfrak{f}) + \overline{N}(r,0;\mathfrak{f})
\end{eqnarray}
Using (\ref{02}) and Lemma (\ref{3.10}), we get
\begin{eqnarray*}
\mu T(r, \mathfrak{f}) & \leq &  \overline{N}(r,a;M[\mathfrak{f}])+ (q_1+1) T(r,\mathfrak{f}) + S(r,\mathfrak{f})
\end{eqnarray*}
Since $ q_1 \geq 1$ and $q_0> q_1+1$, then
 $$T(r, \mathfrak{f})  \leq  \frac{1}{(q_0-q_1-1)}\overline{N}(r,a;M[\mathfrak{f}])+ S(r,\mathfrak{f}).$$ \medbreak
\textbf{Case 2}:  Let $z_0$ be a zero of $\mathfrak{f}(z)$ of multiplicity $q$, but  $\mathfrak{f}(z+c)$ has no zero at $z_0$.\medbreak
If  $\mathfrak{f}'(z+c)$ has no zero at $z_0$, then  $z_0$ is a zero of $\mathfrak{f}(z)^{q_0}(\mathfrak{f}'(z+c))^{q_1}$ of order at least $ qq_0.$ Now,
$$ q\mu +1 -(qq_0)  \leq  qq_1 +1$$
Therefore,
\begin{eqnarray} \label{03}
N(r,0;\mathfrak{f}^{\mu}) + \overline{N}(r,0;M[\mathfrak{f}]) - N(r,0;M[\mathfrak{f}]) \leq  q_1 N(r,0;\mathfrak{f}) + \overline{N}(r,0;\mathfrak{f})
\end{eqnarray}
Using (\ref{03}) and Lemma (\ref{3.10}), we get
\begin{eqnarray*}
\mu T(r, \mathfrak{f}) & \leq &  \overline{N}(r,a;M[\mathfrak{f}])+ (q_1+1) T(r,\mathfrak{f})+ S(r,\mathfrak{f})
\end{eqnarray*}
Since $ q_1 \geq 1$ and $q_0> q_1+1$ then
 $$T(r, \mathfrak{f})  \leq  \frac{1}{(q_0-q_1-1)}\overline{N}(r,a;M[\mathfrak{f}])+ S(r,\mathfrak{f}).$$
 If $\mathfrak{f}'(z+c)$ has a zero at $z_0$ of order $r$, then \medbreak
\textbf{subcase 2.1}:
Let $r=q$, then  $z_0$ is a zero of $\mathfrak{f}(z)^{q_0}(\mathfrak{f}'(z+c))^{q_1}$ of order at least $ qq_0 + rq_1.$ Now,\\ $$ q\mu +1 -(qq_0+ rq_1) \leq  1+q_1 $$
Therefore,
\begin{eqnarray}\label{04}
\nonumber&& N(r,0;\mathfrak{f}^{\mu}) + \overline{N}(r,0;M[\mathfrak{f}]) - N(r,0;M[\mathfrak{f}]) \\
 &\leq &   \overline{N}(r,0;\mathfrak{f})+ q_1\overline{N}(r,0;\mathfrak{f})
\end{eqnarray}
Using (\ref{04}) and lemma (\ref{3.10}), we get 
\begin{eqnarray*}
\mu T(r, \mathfrak{f}) & \leq &  \overline{N}(r,a;M[\mathfrak{f}])+ (q_1+1)T(r,\mathfrak{f})+ S(r,\mathfrak{f})
\end{eqnarray*}
Since $ q_1 \geq 1$ and $q_0> q_1+1$, then
 $$T(r, \mathfrak{f})  \leq  \frac{1}{(q_0-q_1-1)}\overline{N}(r,a;M[\mathfrak{f}])+ S(r,\mathfrak{f}).$$ \medbreak
\textbf{subcase 2.2}: Let $r \neq q$, then  $z_0$ is a zero of $\mathfrak{f}(z)^{q_0}(\mathfrak{f}'(z+c))^{q_1}$ of order at least $ qq_0 + rq_1.$ Now,\\ $$ q\mu +1 -(qq_0+ rq_1) \leq  1+ qq_1 $$
Therefore,
\begin{eqnarray}\label{05}
\nonumber&& N(r,0;\mathfrak{f}^{\mu}) + \overline{N}(r,0;M[\mathfrak{f}]) - N(r,0;M[\mathfrak{f}]) \\
 &\leq &   \overline{N}(r,0;\mathfrak{f})+ q_1 N(r,0;\mathfrak{f})
\end{eqnarray}
Using (\ref{05}) and lemma (\ref{3.10}), we get
\begin{eqnarray*}
\mu T(r, \mathfrak{f}) & \leq &  \overline{N}(r,a;M[\mathfrak{f}])+ (q_1+1)T(r,\mathfrak{f})+ S(r,\mathfrak{f})
\end{eqnarray*}
Since $ q_1 \geq 1$ and $q_0> q_1+1$, then
 $$T(r, \mathfrak{f})  \leq  \frac{1}{(q_0-q_1-1)}\overline{N}(r,a;M[\mathfrak{f}])+ S(r,\mathfrak{f}).$$ \medbreak
\textbf{Case 3:}  Let $\mathfrak{f}(z)$ has no zero at $z_0$, but $\mathfrak{f}(z+c)$ has  zero at $z_0$ of multiplicity $q$.\medbreak
Then  $z_0$ is a zero of $\mathfrak{f}(z)^{q_0}(\mathfrak{f}'(z+c))^{q_1}$ of order at least $ (q-1)q_1$. Now,
$$1 -((q-1)q_1) =  1 -qq_1+q_1	 \leq  q_1 +1$$
Therefore,
\begin{eqnarray} \label{06}
\nonumber && N(r,0;\mathfrak{f}^{\mu}) + \overline{N}(r,0;M[\mathfrak{f}]) - N(r,0;M[\mathfrak{f}]) \\
 &\leq & q_1 \overline{N}(r,0;\mathfrak{f}(z+c)) + \overline{N}(r,0;\mathfrak{f}(z+c))
\end{eqnarray}
Using (\ref{04}), Lemma (\ref{3.2}) and Lemma (\ref{3.10}), we have
\begin{eqnarray*}
\mu T(r, \mathfrak{f}) & \leq &  \overline{N}(r,a;M[\mathfrak{f}])+ q_1 T(r,0;\mathfrak{f}(z+c))+ T(r,0;\mathfrak{f}(z+c))+ S(r,\mathfrak{f})\\
 & \leq &  \overline{N}(r,a;M[\mathfrak{f}])+(q_1+1)T(r, \mathfrak{f})+ S(r,\mathfrak{f}).
\end{eqnarray*}
Since $ q_1 \geq 1$ and $q_0> q_1+1$, then
 $$T(r, \mathfrak{f})  \leq  \frac{1}{(q_0-q_1-1)}\overline{N}(r,a;M[\mathfrak{f}])+ S(r,\mathfrak{f}).$$
This completes the proof of the theorem.
\end{proof}
\section{Acknowledgement}
The Authors are thankful to the referees for their valuable suggestions and comments which considerably improved the presentation of this paper. The authors are thankful to Dr. Bikash Chakraborty for his comments during the preparation of this manuscript.

\end{document}